\documentclass[11pt]{amsart}
\usepackage{geometry}                
\usepackage{graphicx}
\usepackage{amssymb}
\usepackage{epstopdf}
\newcommand{\R}{\mathbb{R}}

\newtheorem{theorem}{Theorem}

\newtheorem{proposition}[theorem]{Proposition}

\newtheorem{lemma}[theorem]{Lemma}
\def\O{{\rm O}}

\def \R{{\rm I\!R}}

\def\cal{\mathcal}

\title{Bounded orbits for 3 bodies in $\mathbb{R}^4$}
\author{Alain Albouy, Holger R.~Dullin} 

\begin{document}

\begin{abstract}
We consider the Newtonian 3-body problem in dimension 4, and fix a value of the angular momentum which is compatible with this dimension. We show that the energy function cannot tend to its infimum on an unbounded sequence of states. Consequently the infimum of the energy is its minimum. This completes our previous work \cite{AD19} on the existence of Lyapunov stable relative periodic orbits in the 3-body problem in $\mathbb{R}^4$.
\end{abstract}

\maketitle

\section{Introduction}

This work aims to continue our previous work \cite{AD19} on the 3-body problem in $\R^4$. There we started from a rather complete description of the configurations of the relative equilibria, which form three curves in the space of triangular shapes. We described their embedding in the phase space, and in particular, the three curves they draw in the energy versus angular momentum diagram. We noticed a cusp on two of the curves, and an interesting connection of the third curve with a fourth curve corresponding to the equilateral relative equilibria. We gave some rigorous results and proposed some difficult conjectures.

The relative equilibria are responsible for changes of topology of the integral manifold. Smale \cite{smale70topology} insisted that changes of topology may also occur at infinity. In this sense, our four curves in the energy versus angular momentum diagram are indicating changes of topology. They must be completed by other curves which indicate the changes of topology which occur at infinity. Different arguments allow us to infer that there are four new curves, and to propose their equation. We complete our diagrams by drawing these conjectural curves. We notice the asymptotic coincidence of three of the new curves and the first three curves of relative equilibria when the energy $H\to -\infty$. This coincidence is easily deduced from the expansions in \cite{DS19}. The fourth curve is $H=0$.
In a future work we will present some more statements and proofs about these ``critical points at infinity''.

As Richard Montgomery warned us, the proof of our most surprising claim in \cite{AD19}, namely {\it For any value of the angular momentum and the masses, there exists in the reduced problem a Lyapunov stable relative equilibrium,} is incomplete. We proved that the energy is bounded below, and that the minimum provides such a stable equilibrium. But we did not prove that the minimum is realized. Here we fill this gap with a quite general argument which proves that a ``critical point at infinity'' cannot realize the infimum of the energy, since it may be approached by inferior values.

This general argument first proposes models for the infimum at infinity, and then proves that these models do not realize the infimum. We conjecture that the ``critical values at infinity'' of the energy are six values deduced from these models. In particular, when the energy is increased, the compact level sets surrounding the Lyapunov stable relative equilibria should lose their compactness. This should happen at one of the six values. The values are generically distinct. They form, when the angular momentum is varied, the three curves with $H<0$ which we add to the energy-momentum diagram.

\section{The infimum is not at infinity in the 4D 3-body problem}

Consider the 3-body problem. Let $q_1$, $q_2$, $q_3$ be the positions and $m_1$, $m_2$, $m_3$ be the masses. The bodies move in a Euclidean vector space $E$. Their center of mass is the origin. Here $E=\R^4$. The vector space $\bigwedge^2 E$ of bivectors or of antisymmetric matrices is also Euclidean. For a simple bivector $\pi=a\wedge b\in \bigwedge^2 E$, the length squared is $|\pi|^2=|a|^2|b|^2-\langle a,b\rangle^2$.

The angular momentum, expressed in terms of the two Jacobi vectors $q=q_2-q_1$ and $Q=q_3-(m_1q_1+m_2q_2)/(m_1+m_2)$ and two positive coefficients 
$$\mu=\frac{m_1m_2}{m_1+m_2},\quad \nu=\frac{m_3(m_1+m_2)}{m_1+m_2+m_3},$$
is $$L= q\wedge p+Q\wedge P,\quad\hbox{where}\quad p=\mu \dot q,\quad P=\nu\dot Q.$$ Recall that $L$ is of rank 4 if and only if $q,p, Q,P$ generate a 4-dimensional space. This means that the motion is not restricted to a 3-dimensional subspace of $E$. The eigenvalues of $L$ are $\pm i l_1, \pm i l_2$, where the two positive numbers $l_1$ and $l_2$ are such that in an orthonormal base $(e_1,e_2,e_3,e_4)$ we have $L=l_1 e_1\wedge e_2+l_2 e_3\wedge e_4$.
Let $d_{ij}=|q_i-q_j|$. The energy is
$$H=T+V,\quad\hbox{with}\quad T=\frac{|p|^2}{2\mu}+\frac{|P|^2}{2\nu}, \quad V=-\frac{m_1m_2}{ d_{12}}-\frac{m_2m_3}{ d_{23}}-\frac{m_1m_3}{ d_{13}}.$$
In \cite{AD19} we gave a short proof of:

\begin{proposition}\label{pro:bounded}
Given three positive masses $m_1, m_2, m_3$ and a $4\times 4$ antisymmetric matrix $L$ of rank 4, consider the 3-body problem in $\mathbb{R}^4$ with these masses, and consider in the phase space the submanifold of states with angular momentum $L$. On this submanifold the energy $H$ is bounded below.
\end{proposition}

We will now prove a statement which constrains the realization of the infimum of $H$.


\begin{proposition}\label{pro:inf} 
Consider an unbounded sequence of states of given angular momentum $L$ of rank 4. If the energy converges to the infimum $H_{\inf}$ of the energy for this $L$, then there is an $(i,j)$, $1\leq i<j\leq 3$ and an $(a,b,A,B)\in E$ with $L=a\wedge b+A\wedge B$ such that $H_{\inf}=H_{\inf}^\infty$ with
$$H_{\inf}^\infty=-\frac{(m_im_j)^3}{2|a\wedge b|^2(m_i+m_j)}.$$
\end{proposition}

\begin{proof} We first prove that the velocities are bounded along the unbounded sequence. Suppose the contrary, i.e., $T\to +\infty$ along a subsequence. Then $T\sim H_{\inf}-V\sim -V$. Apply the map $(q,p,Q,P)\mapsto (2q,p/2,2Q,P/2)$, which preserves $L$. Then, on the image of the subsequence, the energy $T/4+V/2 \sim V/4$ tends to $-\infty$, which contradicts the finite infimum $H_{\inf}$. So, the velocities are bounded and the configuration is unbounded. As $H_{\inf}<0$, according to an easy general lemma (\cite{albouy1993integral}, Lemma 1), we can extract a subsequence such that two positions $q_1$ and $q_2$ (up to renumbering) tend to a limit with respect to their center of mass, and their velocities relative to this center of mass tend to a limit. The distance $|Q|$ of the third body $q_3$ tends to infinity.
Consider any state of the subsequence. There are two easy ways to produce another state with lower energy and same angular momentum.

\begin{enumerate}
\item[(i)]
The momentum $P$ contributes to the angular momentum. Since the contribution is $Q\wedge P$, the component of $P$ along $Q$ does not contribute to the angular momentum. We decrease this component. The energy decreases. We do this for each state of the subsequence. As the angular momentum is finite and $Q$ is infinite, we get a new sequence of states where the velocity $P$ tends to zero. Then the third body does not contribute to the energy in the limit.
\item[(ii)]
Standard considerations about the 2-body problem show that without changing the contribution $q\wedge p$ to the angular momentum, we can decrease the energy by choosing a $(p,q)$ generating a circular motion of the binary. Then the energy is $$H=-\frac{(m_1m_2)^3}{2|q\wedge p|^2(m_1+m_2)}.$$ Taking into account the renumbering, this is $H_{\inf}^\infty$ with $a=q$ and $b=p$.
\end{enumerate}
\end{proof}

We will get a simple lower bound for the factor $|a\wedge b|$ in the denominator of $H_{\inf}^\infty$.

\begin{lemma}\label{lem:Gauss} 
Let  ${\cal C}\subset\bigwedge^2 E$ be the 5-dimensional cone whose points are the bivectors of rank 2. The Gauss map sends ${\cal C}$ into itself.
\end{lemma}

\begin{proof} Note that $0$ is excluded from ${\cal C}$, being of rank 0. The points of ${\cal C}$ are non-singular.  Let the bivector $\eta$ be the image of $\pi\in{\cal C}$ by the Gauss map: $\eta$ is orthogonal to the tangent plane of ${\cal C}$ at $\pi$. There is an orthogonal frame $(f_1,f_2,f_3,f_4)$ of $E$ such that $\pi=f_1\wedge f_2$. By varying $f_1$ and $f_2$ arbitrarily, we see that the tangent plane is generated by $f_1\wedge f_2$, $f_1\wedge f_3$, $f_1\wedge f_4$, $f_2\wedge f_3$, $f_2\wedge f_4$. Then $\eta$ is proportional to $f_3\wedge f_4$, since $\langle f_3\wedge f_4, f_i\wedge f_j\rangle=\langle f_3,f_i\rangle \langle f_4,f_j\rangle-\langle f_4,f_i\rangle \langle f_3,f_j\rangle=0$ if $(i,j)\neq (3,4)$. Consequently $\eta$ is of rank 2.
\end{proof}

The next Lemma is a standard particular case of a statement by Weyl (\cite{weyl1912asymptotische}) about Hermitian matrices, presented and extended in \cite{fulton2000eigenvalues}. The particular case being simple, it deserves a short proof. We deduce it from Lemma~\ref{lem:Gauss}, which improves an argument sketched in \cite{AD19}, page 331.

\begin{lemma}\label{lem:below} 
If a bivector $L=a\wedge b+A\wedge B$ of rank 4 has eigenvalues $\pm i l_1$ and $\pm i l_2$ with $0<l_1\leq l_2$, then $l_1\leq |a\wedge b|$.
\end{lemma}

\begin{proof} Let $d_L$ be the distance from $L=l_1 e_1\wedge e_2+l_2 e_3\wedge e_4$ to the cone ${\cal C}$ of bivectors of rank 2. Let $\pi$ be the point of ${\cal C}$ such that $|L-\pi|=d_L$. The possibility $\pi=0$ is excluded since we would have $d_L=|L|=\sqrt{l_1^2+l_2^2}$, while $\pi=l_2e_3\wedge e_4$ is at the lower distance $l_1$. As $\pi\neq 0$, $L-\pi=\eta$ is orthogonal to the tangent plane of ${\cal C}$ at $\pi$. According to Lemma~\ref{lem:Gauss}, $\eta$ is of rank 2. The decompositions $L=\pi+\eta$ and $L=l_1 e_1\wedge e_2+l_2 e_3\wedge e_4$ coincide if $l_1\neq l_2$, by uniqueness of the eigenplanes. We deduce, even if $l_1=l_2$, that $d_L=l_1$. If now $L=a\wedge b+A\wedge B$, as $A\wedge B\in {\cal C}$, we have $|a\wedge b|\geq d_L$.
\end{proof}

\begin{proposition}\label{pro:sixvalues} 
For any $(i,j)$, $1\leq i<j\leq 3$, for $k=1$ or $2$, there is an unbounded sequence $s_{n}$ of states of given angular momentum $L=l_1 e_1\wedge e_2+l_2 e_3\wedge e_4$, $l_1l_2\neq 0$ such that the energy $H(s_n)$ of $s_n$ satisfies
$$\lim_{n\to\infty} H(s_{n})=H_{ijk}\quad\hbox{and}\quad H(s_{n})<H_{ijk},\quad \hbox{where } H_{ijk}=-\frac{(m_im_j)^3}{2l_k^2(m_i+m_j)}.$$ 
\end{proposition}

\begin{proof} We take for example $(i,j,k)=(1,2,1)$. We choose the following triangle in the plane $\O e_1e_3$, with center of mass at $\O=(0,0)$, and with inertia tensor with axes $\O e_1$ and $\O e_3$:
$$q_1=(\alpha m_2,-\beta_n m_3),\quad 
q_2=(-\alpha m_1,-\beta_n m_3),\quad 
q_3=\bigl(0,\beta_n (m_1+m_2)\bigr).$$
We fix $\alpha>0$ such that the binary $(q_1,q_2)$ has the size of the circular motion of angular momentum $l_1$ and let $\beta_n\to +\infty$.

The velocities $(\dot q_1,\dot q_2,\dot q_3)$ are in the plane $\O e_2 e_4$ orthogonal to the plane of the triangle. We define them as the sum of two components: one component gives an angular momentum $l_2$ to the ``large binary'' consisting of $q_3$ and the center of mass of $q_1$ and $q_2$. The other is fixed and gives to the binary $q_1$, $q_2$ the unique circular motion of angular momentum $l_1$. The total angular momentum is $L=l_1 e_1\wedge e_2+l_2 e_3\wedge e_4$. There is no other component since the rotations are around the axes of inertia.

As $n\to +\infty$, the limit of the energy is $H_{ijk}$. While approaching the limit, we see two contributions tending to zero: a negative contribution $-m_1m_3/d_{13}-m_2m_3/d_{23}$ to the potential energy, a positive contribution to the kinetic energy, proportional to $|\dot q_3|^2$. The negative contribution is dominant, since the
positive contribution is as $1/d_{13}^2$, the angular momentum being finite. So the limit is approached from below.
\end{proof}

\begin{proposition}\label{pro:realized} 
If along a sequence on the submanifold defined in Proposition~\ref{pro:bounded} the energy function converges to its infimum $H_{\inf}$, then the sequence is bounded.
\end{proposition}

\begin{proof}
Consider the infimum $H_{\inf}=H_{\inf}^\infty$ of the energy obtained in Proposition~\ref{pro:inf} by assuming that the sequence is unbounded. Lemma~\ref{lem:below} proves that $H_{ij1}\leq H_{\inf}^\infty$, where $H_{ij1}$ is defined in Proposition~\ref{pro:sixvalues}. But Proposition~\ref{pro:sixvalues} proves that there are states $s_n$ such that $H(s_n)<H_{ij1}$. So, $H_{\inf}^\infty$ is not the infimum. Contradiction.
\end{proof}

In other words, the value $H_{\inf}$ is the value of the energy at a minimum, and is not a limiting value at infinity. We may now confirm the following result, which we stated in \cite{AD19} as the first part of Theorem 9.

\begin{theorem}\label{th:final} 
There is an $\epsilon>0$ such that on the submanifold defined in Proposition~\ref{pro:bounded} the level sets of the energy function $H$ with value $H\in [H_{\inf}, H_{\inf}+\epsilon]$, where $H_{\inf}$ is the infimum of $H$, are nonempty and compact.
\end{theorem}

The second part of Theorem 9 in \cite{AD19} states that the minima of $H$ on the submanifold are, after reduction of the symmetry, isolated relative equilibria, which are consequently Lyapunov stable. There we used a lemma about the balanced configurations to exclude the possibility of a continuum of relative equilibria.

section{Critical values}

\begin{figure} 
\includegraphics[width=8.5cm]{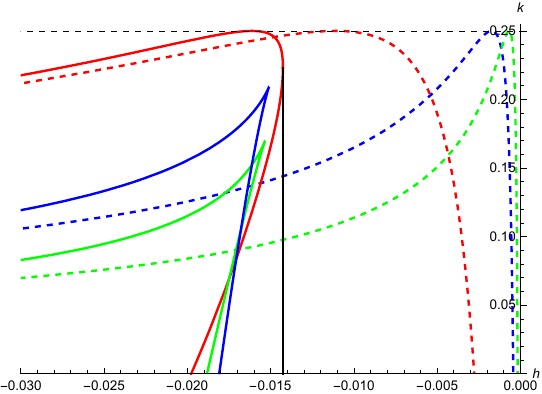} \hspace*{0.5cm} 
\caption{Critical values of the scaled momentum versus scaled energy (as in \cite{AD19}) for finite relative equilibria (solid red, green, blue, black) and for critical points at infinity (dashed). Masses $m_1 = 1/2, m_2 = 1/3, m_3 = 1/6$.}
\label{fig:hkl}
\end{figure}

In \cite{AD19} we found the relation between energy and angular momentum for relative equilibria of 
the 3-body problem in 4D. It is natural to add to the energy-momentum diagram from \cite{AD19} the curves
for critical points at infinity. 
Let $H_{ij} = -m_i^3 m_j^3 / ( 2 ( m_i + m_j)) = H_{ijk} l_k^2$  
and consider the energy multiplied 
by the squared angular momentum to make it scaling invariant as in \cite{AD19}, hence
\[ 
  h = H_{ij} l_1^{-2}  ( l_1 + l_2)^2.
\]
%
Let $h$ be a function of the dimensionless momentum $k = l_1 l_2 / ( l_1 + l_2)^2$.
Introducing the parameter $\chi = (1 + l_2/l_1)^{-1} \in (0, 1)$ we obtain the scaling invariant energy-momentum curves 
of the critical points at infinity as
\[
   (h,k) = ( H_{ij}  \chi^{-2}, \chi ( 1 - \chi) ) .
\]
If we insist that $l_1 \le l_2$ then $\chi \in (0, 1/2]$. Instead of normalising the other way we can consider $\chi \in (0,1)$.

For $\chi \to 1$ the 3D case is recovered, in which the three critical values at infinity are all larger 
than the critical values of the collinear Euler and the equilateral Lagrange solutions.
Considering Fig.~1 this ordering is not preserved for non-zero values of the angular momentum.
At $\chi = 1/2$ the curve of critical values is tangent to the maximum values $k = 1/4$.
For $\chi \le 1/2$ the curves of critical values of the infinite critical points may intersect those
of the finite critical points and the simple ordering of the 3D case is not preserved.
We now show that in the limit $\chi \to 0$ the critical values of the three infinite critical points 
are asymptotic to the critical values of the three families of finite critical points, which are balanced configurations.
In \cite{DS19} a relation between the limit of these balanced configurations and critical values at infinity of the 3D case
has already been noticed, and with the results of the present paper this is made precise.
Eliminate $\chi$ using $k = \chi ( 1 - \chi)$ and choose the branch where $\chi \to 0$.
Expanding $h = h(k)$ in the limit of small $k$ gives
$$
     h = -\frac{ 4 H_{ij}} { (\sqrt{ 1 - 4 k} - 1)^2 } = -\frac{ H_{ij} }{k^2} ( 1 - 2 k - k^2 + O(k^3) ). 
$$
In \cite{DS19} the asymptotics of the three families of balanced configurations has been computed, see page 393.
The series found there agrees in the first two terms so that the limiting value of 
$hk^2$ and its first derivative agree in the limit $k \to 0$. The sign of the difference in the 
2nd order term depends on the masses. In figure~1 solid and dashed lines with the same asymptotics
have the same colour.

\section{Acknowledgement}

The authors would like to thank the referee for suggestions that improved the exposition of the paper.

\appendix

\bibliographystyle{amsalpha}      

\def\cprime{$'$}
\providecommand{\bysame}{\leavevmode\hbox to3em{\hrulefill}\thinspace}
\providecommand{\MR}{\relax\ifhmode\unskip\space\fi MR }
\providecommand{\MRhref}[2]{%
  \href{http://www.ams.org/mathscinet-getitem?mr=#1}{#2}
}
\providecommand{\href}[2]{#2}

\end{document}